\documentclass{amsart}
\usepackage{amscd,amsmath,amssymb,amsthm,bm,cases,color,comment,mathrsfs,mathtools}
\usepackage{graphicx}
\usepackage[all]{xy}

\usepackage{comment}

\usepackage[shortlabels]{enumitem}

\newcommand{\fa}{\mathcal{A}^{\dagger}}
\newcommand{\fna}{\mathcal{NA}^{\dagger}}
\newcommand{\fla}{\mathcal{A}_{\rm Lk}^{\dagger}}

\makeatletter
\@namedef{subjclassname@2020}{%
\textup{2020} Mathematics Subject Classification}
\makeatother
\theoremstyle{definition}
\newtheorem{definition}{Definition}[section]
\newtheorem{remark}[definition]{Remark}

\theoremstyle{plain}
\newtheorem{theorem}[definition]{Theorem}
\newtheorem{proposition}[definition]{Proposition}
\newtheorem{lemma}[definition]{Lemma}
\newtheorem{corollary}[definition]{Corollary}
\numberwithin{equation}{section}
\newtheorem{claim}[definition]{Claim}

\title[Automorphisms of fine curve graphs for nonorientable surfaces]{Automorphisms of fine curve graphs for nonorientable surfaces}

\author[M.~Kimura]{Mitsuaki Kimura}
\address[Mitsuaki Kimura]{Department of Mathematics, Osaka Dental University, 8-1 Kuzuha Hanazono-cho, Hirakata, Osaka 573-1121, Japan}
\email{kimura-m@cc.osaka-dent.ac.jp}

\author[E.~Kuno]{Erika Kuno}
\address[Erika Kuno]{Department of Mathematics,
Graduate School of Science,
Osaka University,
1-1 Machikaneyama-cho,  Toyonaka, Osaka 560-0043, Japan
}
\email{e-kuno@math.sci.osaka-u.ac.jp}

\date{\today}


\begin{document}

\begin{abstract}
The fine curve graph of a surface was introduced by Bowden, Hensel, and Webb as a graph consisting of essential simple closed curves on the surface.
Long, Margalit, Pham, Verberne, and Yao proved that the automorphism group of the fine curve graph of a closed orientable surface is isomorphic to the homeomorphism group of the surface.
In this paper, based on their argument, we prove that the automorphism group of the fine curve graph of a closed nonorientable surface $N$ of genus $g \geq 4$ is isomorphic to the homeomorphism group of $N$.

\end{abstract}

\maketitle

\section{Introduction} \label{Intro}

Let $N=N_{g,p}^{b}$ be a connected nonorientable surface of genus $g\geq 1$ with $b\geq 0$ boundary components and $p\geq 0$ punctures, and $S=S_{g,p}^{b}$ a connected orientable surface of genus $g\geq 0$ with $b\geq 0$ boundary components and $p\geq 0$ punctures. 
Note that $N_{g,p}^{b}$ is homeomorphic to the surface obtained from a sphere with $p$ punctures by removing $g+b$ open disks and attaching $g$ M\"{o}bius bands along their boundaries, and we call each of the M\"{o}bius bands a \emph{crosscap}. 
We drop the subscript $b$ (resp. $p$) from $N_{g,p}^{b}$ and $S_{g,p}^{b}$ if $b=0$ (resp. $p=0$).
In particular, $N_g=N_{g,0}^0$ and $S_g=S_{g,0}^0$ denote closed surfaces.
If we do not care whether the surface is orientable, we write $F$ for the surface. 

In \cite{Bowden--Hensel--Webb22}, Bowden, Hensel, and Webb introduced the {\it fine curve graph} $\mathcal{C}^{\dagger}(F)$ of a surface $F$ in order to study the homeomorphism group $\mathrm{Homeo}(F)$ and the diffeomorphism group $\mathrm{Diff}(F)$ of $F$. Here, the fine curve graph $\mathcal{C}^{\dagger}(F)$ of a surface $F$ is the graph whose vertices are essential simple closed curves in $F$ and whose edges correspond to pairs of vertices that are disjoint in $F$.
The action of $\mathrm{Homeo}(F)$ on $\mathcal{C}^{\dagger}(F)$ induces the natural map $\eta\colon\mathrm{Homeo}(F)\rightarrow\mathrm{Aut}(\mathcal{C}^{\dagger}(F))$.

In \cite{Long--Margalit--Pham--Verberne--Yao}, Long, Margalit, Pham, Verberne and Yao proved that the natural map $\eta\colon\mathrm{Homeo}(S_{g})\rightarrow\mathrm{Aut}(\mathcal{C}^{\dagger}(S_{g}))$ is an isomorphism for $g\geq 2$.
In this paper, we extend their result to nonorientable surfaces. Namely, we prove the following:

\begin{theorem}\label{thm:main}
For $g\geq 4$, the natural map $\eta\colon\mathrm{Homeo}(N_{g})\rightarrow\mathrm{Aut}(\mathcal{C}^{\dagger}(N_{g}))$ is an isomorphism.
\end{theorem}



The result of \cite{Long--Margalit--Pham--Verberne--Yao} can be thought of as an analogy to the classical result of Ivanov~\cite{Ivanov97}, which states that the automorphism group ${\rm Aut}(\mathcal{C}(S_g))$ of the ordinary curve graph $\mathcal{C}(S_g)$ is isomorphic to the extended mapping class group $\mathrm{Mod}^{\pm}(S_g)$ of $S_g$ if $g \geq 3$. Similarly, Theorem~\ref{thm:main} can also be thought of as an analogy to the result of Atalan and Korkmaz~\cite{Atalan--Korkmaz14}, which states that the automorphism group ${\rm Aut}(\mathcal{C}(N_g))$ is isomorphic to the mapping class group $\mathrm{Mod}(N_g)$ if $g \geq 5$.

To prove Theorem \ref{thm:main}, we apply the argument in \cite{Long--Margalit--Pham--Verberne--Yao}, but with some modifications.
For nonorientable surfaces, we need to consider the following differences:


\begin{itemize}
\item Not only two-sided curves but also one-sided curves appear on a nonorientable surface. Thus, it is necessary to properly define whether one-sided curves are allowed or not for several concepts such as for torus pairs, pants pairs, and bigon pairs (see Subsection~\ref{notations}). 
We also observe that automorphisms of $\mathcal{C}^{\dagger}(N)$ preserve two-sidedness of curves (Lemma~\ref{presv_two-sided}).

\item Inessential simple closed curves of nonorientable surfaces consist not only of curves bounding a disk but also of curves bounding a M\"obius band. It affects the proofs of Lemmas~\ref{presv_torus_pairs} and \ref{presv_annulus_sets} for example. 
We also have to consider whether or not curves bounding a M\"obius band are allowed  in our definition of the extended fine curve graph $\mathcal{EC}^{\dagger}(N)$ (see Remark~\ref{remark_for_extended_graph}).
\end{itemize}

In related research, Roux and Wolff~\cite{Roux--Wolff22} considered a variant $\mathcal{NC}_{\pitchfork}^{\dagger}(F)$ of the fine curve graph and its automorphism group. They proved that ${\rm Aut}(\mathcal{NC}_{\pitchfork}^{\dagger}(F))$ is isomorphic to ${\rm Homeo}(F)$ for every nonspherical surface (i.e., surfaces not embeddable in the 2-sphere) without boundary, orientable or not, compact or not.
In \cite{Roux--Wolff22}, they mentioned the possibility of proving 
the result of \cite{Long--Margalit--Pham--Verberne--Yao} 
via the graph $\mathcal{NC}_{\pitchfork}^{\dagger}(F)$. 
We hope that Theorem \ref{thm:main} could also be approached with the same strategy.

\section{Preliminaries}\label{preliminaries}

\subsection{Nonorientable surfaces}\label{notations}

Throughout this paper, a curve on $F$ means a simple closed curve on $F$ unless otherwise noted.
A curve $c$ on a surface $F$ is said to be {\it one-sided} if a regular neighborhood of $c$ is a M\"obius band, and $c$ is said to be {\it two-sided} if a regular neighborhood of $c$ is an annulus. 
Every curve in an oriented surface is two-sided.
We remark that a curve $c$ on $N$ is one-sided if and only if $c$ goes through crosscaps odd number of times, and two-sided if and only if $c$ goes through crosscaps even number of times.

For an orientable surface $S=S_g^b$ with $g\geq 1$, 
the surface obtained by cutting $S$ along a nonseparating curve is homeomorphic to $S_{g-1}^{b+2}$.
For a nonorientable surface $N=N_g^b$ with $g\geq 1$, let $c$ a nonseparating curve in $N$, and let $F$ denote the surface obtained by cutting $N$ along $c$. Then, 
    \begin{itemize}
        \item $F$ is homeomorphic to $N_{g-1}^{b+1}$ or $S_{\frac{g-1}{2}}^{b+1}$ if $c$ is one-sided, and
        \item $F$ is homeomorphic to $N_{g-2}^{b+2}$ or $S_{\frac{g-2}{2}}^{b+2}$ if $c$ is two-sided.
    \end{itemize}

The following lemma, which will be used in the proof of Lemma \ref{presv_two-sided}, can be seen immediately from the above facts.

\begin{lemma}\label{lem:genus_cutting}
 Let $c$ be a nonseparating curve in a surface $F_g^b$ with $g\geq 1$.
 Assume that the surface obtained by cutting $F_g^b$ along $c$ is homeomorphic to $F_{g'}^{b'}$. Then $g > g'$ holds. 
 Moreover, if $F_g^b$ is nonorientable and $g-g'=1$, then $c$ is one-sided.
\end{lemma}

We can also observe the following fact. We will use it to prove Lemma~\ref{presv_annulus_sets}~(1).

\begin{lemma}\label{lem:cutting_separate}
 Let $a$ and $b$ be nonseparating curves in a surface $N_g$ with $g\geq 3$.
 Assume that the surface obtained by cutting $N_g$ along $a$ and $b$ has two connected components $F$ and $F'$. 
Then, there exists a separating essential curve in $N_g$ that lies in $F$ if and only if $F$ is homeomorphic to neither $S_0^2$ nor $N_1^2$.
\end{lemma}

\subsection{Torus pairs, pants pairs and bigon pairs}

We define several concepts that appear in \cite{Long--Margalit--Pham--Verberne--Yao} for nonorientable surfaces.
We say that curves $c$ and $d$ on $N$ are {\it noncrossing at a component $a$}  of $c\cap d$ if there are a neighborhood $U$ of $a$ and a homeomorphism $U\rightarrow\mathbb{R}^{2}$ so that the image of $c\cap U$ and $d\cap U$ lie in the (closed) upper and lower half plane of $\mathbb{R}^{2}$, respectively (otherwise, $c$ and $d$ are called {\it crossing at $a$}).
The curves $c$ and $d$ are {\it noncrossing} if they are noncrossing at any component of $c\cap d$. 

Let $c$ and $d$ be curves on $N$. We say that a pair $\{c, d\}$ of essential simple closed curves in $N$ is:

\begin{itemize}
\item a {\it torus pair} if $c$ and $d$ are both two-sided, $c\cap d$ is a single interval, and $c$ and $d$ are crossing at that interval.
\item a {\it pants pair} if $c$ and $d$ are both two-sided, $c\cap d$ is a single interval, $c$ and $d$ are noncrossing at that interval, and $c$ and $d$ are not homotopic.
\item {\it bigon pair} if $c$ and $d$ are both two-sided, $c\cap d$ is just one nontrivial closed interval, and $c$ and $d$ are homotopic.
\end{itemize}


See Remark \ref{obstruction_in_step_1} for a discussion of why the definitions of torus pairs, pants pairs, and bigon pairs are restricted to two sided curves.

A torus pair or a pants pair $\{c, d\}$ is {\it degenerate} if $c\cap d$ is a single point.
If $\{c, d\}$ is a nondegenerate torus pair, then the curve $\overline{c\bigtriangleup d}\coloneqq \overline{c\cup d - c\cap d}$ is the exactly one other essential simple closed curve $e$ which is contained in $c\cup d$, and we refer to $\{c,d,e\}$ as a {\it torus triple}. We remark that if $\{c, d\}$ is a nondegenerate torus pair, then the curve $\overline{c\bigtriangleup d}$ is two-sided. 
We can see this as follows: If the arc $c\cap d$ passes the crosscaps odd (resp. even) number of times, then both arcs $c-c\cap d$ and $d-c\cap d$ also pass crosscaps odd (resp. even) number of times since $c$ and $d$ are both two-sided curves. Then we see that the curve $\overline{c\bigtriangleup d} = \overline{(c-c\cap d) \cup (d-c\cap d)}$ is also two-sided since it passes the crosscaps even number of times. 

If $\{c, d\}$ is a bigon pair, then $e\coloneqq\overline{c\bigtriangleup d}$ is a simple closed curve bounding a disk. When the two curves in a bigon pair are nonseparating, we refer the pair to a {\it nonseparating bigon pair}.

Moreover, we suppose that bigon pairs $\{c,d\}$ and $\{c', d'\}$ determine the same inessential curve $e$ bounding a disk. We say that the pair of bigon pairs $\{\{c,d\}, \{c', d'\}\}$ is a {\it sharing pair} if the corresponding arcs connecting $e$ have disjoint interiors. The sharing pair is {\it linked} if the corresponding arcs are linked at $e$, which means that a curve parallel to $e$ and sufficiently close to $e$ intersects the arcs alternately.   

\subsection{Graph theory}
In the next section, we will consider claims of the type that isomorphisms of fine curve graphs preserve certain properties of curves. 
The proof of these claims is made by reformulating the topological conditions for curves into graph-theoretic conditions for fine curve graphs. Here, we summarize the notion of join and link that will appear later.

A graph is a {\it join} if its vertices are partitioned into two or more nonempty sets such that each pair of vertices between different sets is connected by an edge. 

The {\it link} of a set $A$ of vertices in a graph is the subgraph spanned by the set of vertices that are not in $A$ and are connected by an edge to each vertex in $A$. For example, for $c,d_1,\dots,d_k \in \mathcal{C}^{\dagger}(F)$, $c$ is a vertex of the link of $\{d_1,\dots,d_k\}$ in $\mathcal{C}^{\dagger}(F)$ if and only if $c$ and $d_i$ are disjoint for each $i=1,\dots,k$.

\section{Key propositions}

In this section, we provide an outline of the following key propositions for the proof of Theorem~\ref{thm:main}.



\begin{proposition}[cf. {\cite[Proposition 2.1]{Long--Margalit--Pham--Verberne--Yao}}]\label{presv_bigon_pairs}
Let $N=N_{g}$ with $g\geq 4$. Then every automorphism of $\mathcal{C}^{\dagger}(N)$ preserves the set of nonseparating bigon pairs.
\end{proposition}

\begin{proposition}[cf. {\cite[Proposition 2.2]{Long--Margalit--Pham--Verberne--Yao}}]\label{presv_linked_sharing_pairs}
Let $N=N_{g}$ with $g\geq 4$. Then every automorphism of $\mathcal{C}^{\dagger}(N)$ preserves the set of linked sharing pairs.
\end{proposition}

 We also list the lemmas required for the key propositions.
In this subsection, all propositions and lemmas except Lemma \ref{presv_two-sided} are taken from \cite[Section 2]{Long--Margalit--Pham--Verberne--Yao}, with the assumption of orientable surfaces changed to nonorientable surfaces. 


A {\it multicurve} is a finite collection of pairwise disjoint essential simple closed curves in $N$. A multicurve is {\it separating} if its complement has more than one component. We say that two curves $a$ and $b$ lie on the {\it same side} of a separating multicurve $m$ if they are disjoint from $m$ and lie in the same complementary component. We remark that a separating multicurve $m$ should contain at least one two-sided curve. 

\begin{lemma}[cf. {\cite[Lemma 2.3]{Long--Margalit--Pham--Verberne--Yao}}]\label{presv_separating_multi_curves}
Let $N=N_{g}$ with $g\geq 3$. Then every automorphism $\alpha$ of $\mathcal{C}^{\dagger}(N)$ preserves:
\begin{itemize}
\item the set of separating curves in $\mathcal{C}^{\dagger}(N)$, and
\item the set of separating multicurves in $\mathcal{C}^{\dagger}(N)$.
\end{itemize}
Moreover, $\alpha$ preserves the sides of separating multicurves, that is, $a$ and $b$ lie on the same side of a separating multicurve $m$ if and only if $\alpha(a)$ and $\alpha(b)$ lie on the same side of $\alpha(m)$. 
\end{lemma}

Lemma~\ref{presv_separating_multi_curves} can be proved as in \cite{Long--Margalit--Pham--Verberne--Yao} in the following way: 
We can observe that a multicurves $m$ is separating if and only if the link of $m$ is a join; this shows the former claim.
    The latter claim follows from the fact that two curves are on the same side of $m$ if and only if they belong to the same set in the partition of the link of $m$ as a join.


We define the {\it hull} of a collection of curves in a surface to be the union of the curves along with any embedded disks bounded by the curves. 


\begin{lemma}[cf. {\cite[Lemma 2.4]{Long--Margalit--Pham--Verberne--Yao}}]\label{presv_hulls}
Let $N=N_{g}$ with $g\geq 2$, $X$ a finite set of vertices represented by two-sided curves on $N$ of $\mathcal{C}^{\dagger}(N)$, and $d$ a vertex of $\mathcal{C}^{\dagger}(N)$. Then, $d$ lies in the hull of $X$ if and only if the link of $d$ contains the link of $X$. 
In particular, if $d$ lies in the hull of $X$, then $\alpha(d)$ lies in the hull of $\alpha(X)$ for every $\alpha \in \mathrm{Aut}(\mathcal{C}^{\dagger}(N))$.
\end{lemma}

Note that we assume that $X$ consists only of two-sided curves in Lemma \ref{presv_hulls}. With this restriction, we can prove Lemma \ref{presv_hulls} as same as the proof of \cite[Lemma 2.4]{Long--Margalit--Pham--Verberne--Yao}. 
We include a proof of Lemma \ref{presv_hulls} to explain why we added such a restriction (Remark \ref{remark_presv_hulls}).

\begin{proof}[Proof of Lemma~\ref{presv_hulls}]
For the forward direction, suppose that $d$ is a vertex of $\mathcal{C}^{\dagger}(N)$ that lies in the hull of $X$. We assume that there exists a vertex $e$ of $\mathcal{C}^{\dagger}(N)$ which intersects $d$ but is disjoint from each curve in $X$. Then, $e$ should be contained in $N\setminus X$. Since $d$ lies in the hull of $X$ and $e$ intersects with $d$, $e$ should be contained in a connected component of $N\setminus X$ which is a disk. Then it follows that $e$ is an inessential curve bounding a disk. It contradicts that $e$ is a vertex of $\mathcal{C}^{\dagger}(N)$.

Suppose for the other direction that $d$ is a vertex of $\mathcal{C}^{\dagger}(N)$ that does not lie in the hull of $X$ (i.e., $d$ is not a vertex of $X$ nor a curve lies in the union of disks bounding the curves in $X$). This means that there is a component of $d\setminus X$ that lies in a component $R$ of $N\setminus X$ that is not a disk. If the genus (as either orientable or nonorientable surface) of $R$ is at least $1$, it is clear that there exists a curve on $R$ which is essential in $N$. If the genus of $R$ is $0$, then each curve in $R$ parallel to a boundary component of $R$ is an essential curve (here we use two-sidedness of the curves corresponding to the vertices of $X$). Particularly, $R$ contains an essential curve in $N$ which intersects $d$, as desired.
\end{proof}

\begin{remark}\label{remark_presv_hulls}
We note that in Lemma~\ref{presv_hulls}, the set $X$ should consist of only vertices represented by two-sided curves in $N$; if we allow one-sided curves, the same proof of Lemma~\ref{presv_hulls} (\cite[Lemma 2.4]{Long--Margalit--Pham--Verberne--Yao}) does not work. More specifically, if $X$ contains a one-sided curve $c$ and the genus of $R$ in the proof of Lemma~\ref{presv_hulls} is $0$, then there is the boundary component $\partial_{c}$ of $N\setminus X$ corresponding to $c$. Thus a curve on $R$ parallel to $\partial_{c}$ bounds a M\"obius band in $N$, which is not essential in $N$.
We also remark that we will use Lemma~\ref{presv_hulls} to the proof of Lemma~\ref{presv_torus_pairs} as $X$ is a torus pair or a pants pair, and so this restriction does not disturb the argument in this paper.
\end{remark}

We will need the following lemma in our proof of Lemma \ref{presv_torus_pairs}.

\begin{lemma}\label{presv_two-sided}
Let $N=N_{g}$ with $g\geq 4$, $\alpha$ an automorphism of $\mathcal{C}^{\dagger}(N)$, and $c$ a one-sided (resp. two-sided) curve. 
Then $\alpha(c)$ is also a one-sided (resp. two-sided) curve.
\end{lemma}

\begin{proof}
    Let $c$ be an one-sided curve in $N$. We prove that $\alpha(c)$ is also one-sided. Let $F$ and $F'$ be surfaces obtained by cutting $N_g$ along $c$ and $\alpha(c)$, respectively. 
    
    If $F$ is orientable, then $F$ is homeomorphic to $S^1_{\frac{g-1}{2}}$. Assume that $\alpha(c)$ is two-sided. Then $F'$ is homeomorphic to $N^2_{g-2}$ since $g$ is odd. Thus we can take a nonseparating multicurve $\{d_1,\cdots, d_{g-2}\}$ in $F'$. This implies that $\{\alpha(c),d_1,\cdots, d_{g-2}\}$ is a nonseparating multicurve in $N_g$. By Lemma~\ref{presv_separating_multi_curves}, 
    $\{c,\alpha^{-1}(d_1),\cdots, \alpha^{-1}(d_{g-2})\}$ is also a nonseparating multicurve in $N_g$, and it implies $F$ admits a non separating multicurve $\{\alpha^{-1}(d_1),\cdots, \alpha^{-1}(d_{g-2})\}$ consisting of $(g-2)$ curves. Hence we obtain $\frac{g-1}{2} \geq g-2$ but this contradicts to the assumption $g \geq 4$. Therefore $\alpha(c)$ is one-sided.

    If $F$ is nonorientable, then $F$ is homeomorphic to $N_{g-1}^1$. 
    Thus we can take a nonseparating multicurve $\{d_1,\cdots, d_{g-1}\}$ in $F$. We can see that $F$ admits a nonseparating multicurve consisting of $(g-1)$ curves by the same argument above.  
    Hence the genus of $F$ is $(g-1)$, and therefore $\alpha(c)$ is one-sided by Lemma~\ref{lem:genus_cutting}.
\end{proof}

From now on, we prove Lemma \ref{presv_torus_pairs}. We note that there are several difference from the proof of \cite[Lemma 2.5]{Long--Margalit--Pham--Verberne--Yao} (see Remark~\ref{obstruction_in_step_1} for the reason why the same argument as that of orientable surfaces does not work).


\begin{lemma}[cf. {\cite[Lemma 2.5]{Long--Margalit--Pham--Verberne--Yao}}]\label{presv_torus_pairs}
Let $N=N_{g}$ with $g\geq 4$. Then every automorphism $\alpha$ of $\mathcal{C}^{\dagger}(N)$ preserves:
\begin{itemize}
\item the set of torus pairs,
\item the set of degenerate torus pairs,
\item the set of nondegenerate torus pairs, and
\item the set of torus triple.
\end{itemize}
\end{lemma}

\begin{proof}

As in the case of orientable surfaces \cite{Long--Margalit--Pham--Verberne--Yao}, we can prove Lemma~ \ref{presv_torus_pairs} in four steps:

\begin{enumerate}[1.]
    \item $\alpha$ preserves the union of torus pairs and pants pairs consists of only nonseparating curves
    \item $\alpha$ preserves the set of torus pairs
    \item $\alpha$ preserves the degenerate torus pairs
    \item $\alpha$ preserves the set of torus triples
\end{enumerate}

{\it Step 1.}
Let $c$ and $d$ are nonseparating two-sided curves which intersect.
It is enough to show that the following three conditions are equivalent:
\begin{enumerate}[(1)]
\item The pair $\{c,d\}$ is a torus pair or a pants pair.
\item There exists at most one other vertex of $\mathcal{C}^{\dagger}(N)$ lies in the hull of $\{c,d\}$.
\item There exists at most one vertex of $\mathcal{C}^{\dagger}(N)$ whose link contains the link of $\{c,d\}$. 
\end{enumerate}

Step 1 follows from the equivalence of (1)--(3) and Lemma \ref{presv_two-sided}.
 Items (2) and (3) are equivalent by Lemma \ref{presv_hulls}.
 Item (1) implies (2) since if $\{c,d\}$ is a torus pair or a pants pair, then its hull is $c \cup d$.
Now we prove that (2) does not hold if (1) does not (this means that (2) implies (1)).

We modify the argument in \cite{Long--Margalit--Pham--Verberne--Yao}.
Assume that (1) does not hold (i.e., $\{c, d\}$ is neither a torus pair nor a pants pair).
If $\{c, d\}$ is a bigon pair, then the hull of $\{c, d\}$ admits infinitely many vertices, which means that (2) does not hold. 
Thus, it is sufficient to consider the case that $\{c, d\}$ is not a bigon pair. 
Then, $c \cap d$ has two distinct connected components, say $a_1$ and $a_2$. 
They divide $c$ into two curves $c_1$ and $c_2$: $c \setminus (a_1 \sqcup a_2) = c_1 \sqcup c_2$. 
Similarly, they also divide $d$ into $d_1$ and $d_2$
We can take $a_1$ and $a_2$ so that $c \cap d_1 = \emptyset$
(for example, if we consider a moving point on $d$ starting from a component $a_1$, we can take $a_2$ as the component where the point hits $c$ for the first time).
We also take two distinct connected components $b_1$ and $b_2$ of $c \cap d$ with $a_1 \neq b_1$ (it could be $a_1 = b_2$ or $a_2 = b_1$). 
They divide $c$ into $c'_1,c'_2$ and divide $d$ into $d'_1,d'_2$.
We can take $b_1$ and $b_2$ to satisfy $c'_1 \cap d = \emptyset$ and $c'_1 \neq c_1$ (we can choose $c'_1 = c_2$ if $a_1 = b_2$ and $a_2 = b_1$). 
In $c \cup d$, there are four distinct simple closed curves $e_1$, $e_2$, $e_3$, and $e_4$ that contain $c_1 \cup d_1$, $c_2 \cup d_1$, $c'_1 \cup d'_1$, and $c'_2 \cup d'_1$, respectively; they are all distinct from $c$ and $d$. 

If we assume that (2) holds (for contradiction), then the following holds: at least one of pairs $\{e_1,e_2\}$ and $\{e_3,e_4\}$ consists only of inessential curves; let $e_1$ and $e_2$ be inessential without loss of generality. Then, we observe contradictions as follows:
\begin{itemize}
    \item If $e_1$ or $e_2$ bound a disk, then the hull of $\{c,d\}$ contains infinitely many distinct curves; this contradicts the assumption that (2) holds.
    \item If both $e_1$ and $e_2$ bound a M\"obius band, then they represent the same homology class $[e_1]=[e_2]$ of $H_1(N,\mathbb{Z}/2\mathbb{Z})$. Then $[c]=[e_1]+[e_2] = 0$ in $H_1(N,\mathbb{Z}/2\mathbb{Z})$; this means that $c$ is separating, which is contrary to the assumption that $c$ is nonseparating.
\end{itemize} Therefore, we have finished the proof of Step 1.

\medskip

{\it Step 2.}
By Lemma~\ref{presv_separating_multi_curves}, it is sufficient to prove the following.
Let $c,d$ are nonseparating two-sided curves which form a torus pair or a pants pair.
Then, the pair $\{c,d\}$ is a torus pair if and only if there exists a separating curve $e$ such that $e$ is disjoint from both $c$ and $d$ and satisfies the following property: all nonseparating simple closed curves in $N$ lying on the same side of $e$ as $\{c,d\}$ intersect $c\cup d$.     
We can prove it in the same way as the proof of \cite[Step 2 in Lemma 2.5]{Long--Margalit--Pham--Verberne--Yao} as follows.


Let $\{c,d\}$ be a torus pair (then it automatically follows that both $c$ and $d$ are nonseparating). We put $R$ as a neighborhood of $c\cup d$ which is homotopic to a torus with one boundary component $e$ (here, we use the two-sidedness of torus pairs). The resulting surface by cutting $R$ along $c\cup d$ is an annulus. Any nonseparating curve in $R$ is not homotopic to $e$. Hence, any nonseparating curve in $R$ intersects with $c\cup d$. For the other direction, we assume that $c,d$ are nonseparating two-sided curves on $N$ form a pats pair $\{c,d\}$. Let $e$ be any separating curve which is disjoint from $c\cup d$, and let $R$ be the subsurface of $N$ that contains $c\cup d$ and has boundary $e$. Then, $R$ need to have a positive genus, since now $N$ is closed and $c$ and $d$ is not homotopic. There exists a closed neighborhood of $c\cup d$ which is a pair of pants $P$ contained in $R$. We put $P^{\circ}$ as the interior of $P$. Since the genus of $P$ is $0$, there exists a curve on $R\setminus P^{\circ}$ which is nonseparating in $R$ (hence in $N$), and so we are done.    

Step 3 and Step 4 can be proved as in \cite{Long--Margalit--Pham--Verberne--Yao} as follows and thus we obtain Lemma~\ref{presv_torus_pairs}: We can observe that a torus pair $\{c,d\} $ is nondegenerate if and only if there is exactly one other vertex of $\mathcal{C}^{\dagger}(N)$ in the hull of $\{c,d\}$;
Step 3 follows from this observation and Lemma \ref{presv_hulls}. 
Step 4 follows from Step 2, Lemma \ref{presv_hulls} and the fact that if $\{c,d,e\}$ is a torus triple, then $e$ is the exactly one other vertex in the hull of $\{c,d\}$. 
\end{proof}

\begin{remark} \label{obstruction_in_step_1}
We summarize here the issues in considering nonorientable surfaces in the proof of Lemma \ref{presv_torus_pairs}.
\begin{enumerate}
    \item Since we imposed two-sidedness on torus pairs and pants pairs, we needed Lemma~\ref{presv_two-sided} to deduce the conclusion from (1)--(3) in Step 1.
    \item In Step 1 of the proof of Lemma~\ref{presv_torus_pairs}, 
    we argued under the assumption that torus pairs and pants pairs consist of only nonseparating curves, which did not need for the case of orientable surfaces \cite{Long--Margalit--Pham--Verberne--Yao}. 
    Without that assumption, if we attempt to make the same argument as in the orientable case, we have the following issue:
    In \cite{Long--Margalit--Pham--Verberne--Yao}, they used the fact that any inessential curve of orientable surfaces bounds a disk, and so if at least one curve in $\{e_{1},e_{2}, e_{3}, e_{4}\}$ is inessential, then it follows that the hull of $\{c,d\}$ contains infinite essential curves other than $c$ and $d$. However, for nonorientable surfaces, the curves bounding a M\"obius band are also inessential. Then we have the case where the hull of $\{c,d\}$ contains at most one essential curve other than $c$ and $d$ but $\{c,d\}$ does not form a torus pair nor a pants pair if we allow that $c$ or $d$ to be separating. 
    \item The reason why we consider only two-sided curves for torus pairs/triples is as follows: if both $c$ and $d$ are one-sided curves in Step 2, 
    the curve $e$ may not be separating by taking a boundary curve of $R$.
\end{enumerate}

\end{remark}

We explain the definition of an annulus set from now on. Suppose that $(a,b)$ is an ordered pair of vertices in $\mathcal{C}^{\dagger}(N)$ that are disjoint, homotopic two-sided curves. Assuming $g\geq 3$, there is a unique annulus $A$ in $N_{g}$ whose boundary is $a\cup b$. Let $\mathcal{C}^{\dagger}(a,b)$ be the set of vertices of $\mathcal{C}^{\dagger}(N_{g})$ consisting of curves contained in the interior of $A$. We refer to $\mathcal{C}^{\dagger}(a,b)$ as an {\it annulus set}. 
An {\it annulus pair} is an element of an annulus set. 
A {\it nonseparating noncrossing annulus pair} is an annulus pair where both curves are nonseparating and the pair is noncrossing. There is a natural partial ordering on the annulus set $\mathcal{C}^{\dagger}(a,b)$: we say that $c\preceq d$ if $c$ and $d$ are noncrossing and each component of $c\setminus d$ lies in the component of $A\setminus d$ bounded by $a$.

\begin{lemma}[cf. {\cite[Lemma 2.6]{Long--Margalit--Pham--Verberne--Yao}}]\label{presv_annulus_sets}
Let $N=N_{g}$ with $g\geq 4$, let $\alpha$ be an automorphism of $\mathcal{C}^{\dagger}(N)$, and let $a$ and $b$ be disjoint, homotopic nonseparating two-sided curves. Then,
\begin{enumerate}[\textup{(}$1$\textup{)}]
    \item the curves $\alpha(a)$ and $\alpha(b)$ are disjoint, homotopic nonseparating two-sided curves,
    \item the image of $\mathcal{C}^{\dagger}(a,b)$ under $\alpha$ is $\mathcal{C}^{\dagger}(\alpha(a), \alpha(b))$,
    \item if $c,d\in\mathcal{C}^{\dagger}(a,b)$ are noncrossing then $\alpha(c)$ and $\alpha(d)$ are noncrossing, and
    \item if $c\preceq d$ in $\mathcal{C}^{\dagger}(a,b)$ then $\alpha(c)\preceq\alpha(d)$ in $\mathcal{C}^{\dagger}(\alpha(a), \alpha(b))$.
\end{enumerate}
\end{lemma}

Note that if two curves $a$ and $b$ on $N$ are disjoint and homotopic, then it automatically follows that $a$ and $b$ are two-sided since any pair of homotopic one-sided curves must intersect.
We give a proof of Lemma~\ref{presv_annulus_sets}~(1) since we need a modification from the orientable case \cite[Lemma 2.6]{Long--Margalit--Pham--Verberne--Yao}.

\begin{proof}[Proof of Lemma~\ref{presv_annulus_sets}~(1)]
For two disjoint nonseparating two-sided curves $a,b$ such that $\{a,b\}$ is a separating multicurve, we can observe, by Lemma~\ref{lem:cutting_separate}, that the following are equivalent:
\begin{itemize}
    \item at least one component of the surface obtained by cutting $N_g$ along $a$ and $b$ is homeomorphic $S_0^2$ or $N_1^2$,
    \item all separating curves disjoint from $a$ and $b$ lie on the same side of $\{a,b\}$ (if they exist).
\end{itemize}
Let $F$ and $F'$ (resp. $\tilde{F}$ and $\tilde{F}'$) denote the connected components of the surface obtained from $N_g$ by cutting along $a$ and $b$ (resp. $\alpha(a)$ and $\alpha(b)$).
Since $a$ and $b$ are homotopic, we can assume that $F \cong S_0^2$ and $F' \cong N_{g-2}^2$.
By Lemma \ref{presv_separating_multi_curves} and the above observation, $\tilde{F} \sqcup \tilde{F}'$ is homeomorphic to $  S_0^2 \sqcup N_{g-2}^2$ or $N_1^2\sqcup N_{g-3}^2$. Let us assume that the latter holds (to derive a contradiction). We assume that $\tilde{F} \cong N_1^2$. Take an one-sided curve $c$ lies in $\tilde{F}$. Since $\alpha^{-1}(c)$ is one-sided by Lemma~\ref{presv_two-sided}, $\alpha^{-1}(c)$ must lie in $F'$. If we take a separating essential curve $d$ that lies in $F'$ (it exists since $g \geq 4$), then $\alpha(d)$ is separating and lies in $\tilde{F}$ by Lemma \ref{presv_separating_multi_curves}.
This contradicts Lemma~\ref{lem:cutting_separate} and hence (1) follows. 
\end{proof}

Item (2) follows from (1) and Lemma \ref{presv_separating_multi_curves}.
In (3) and (4) we discuss in an annular domain, and it does not depend on the orientability of the surface. Therefore, they are shown by exactly the same arguments as in \cite{Long--Margalit--Pham--Verberne--Yao}, so we omit their proofs.

    

Now we define type 1 and type 2 curves. Suppose that $\{c,d\}$ is a nonseparating noncrossing annulus pair, and suppose that $e$ is a curve so that $\{c,e\}$ and $\{d, e\}$ are degenerate torus pairs. If $c\cup e$ and $d\cap e$ are the same point, then we say that $e$ is a {\it type 1 curve} for $\{c,d\}$. Otherwise, we say that $e$ is a {\it type 2 curve} for $\{c,d\}$.

\begin{lemma}[cf. {\cite[Lemma 2.7]{Long--Margalit--Pham--Verberne--Yao}}]\label{presv_type12}
Let $N=N_{g}$ with $g\geq 3$. Then every automorphism $\alpha$ of $\mathcal{C}^{\dagger}(N)$ preserves type 1 and type 2 curves for nonseparating noncrossing annulus pairs. More precisely, if $\{c,d\}$ is a nonseparating noncrossing annulus pair and $e$ is a type 1 curve  for $\{c,d\}$, then $\alpha(e)$ is a type 1 curve for the nonseparating noncrossing annulus pair $\{\alpha(c),\alpha(d)\}$, and similar for type 2 curves.
\end{lemma}

The proof of Lemma \ref{presv_type12} is again discussed in an annular domain and does not depend on the orientability of the surface, so we omit the proof.


At the end of the section, we discuss how Propositions~\ref{presv_bigon_pairs} and \ref{presv_linked_sharing_pairs} can be derived from the lemmas in this section.
For Proposition~\ref{presv_bigon_pairs}, each nonseparating bigon pair is a nonseparating noncrossing annulus pair $\{c,d\}$ that forms exactly one inessential curve bounding a disk and has the additional property that $c\cap d$ is a nondegenerate interval. Hence, by Lemmas~\ref{presv_annulus_sets} and \ref{presv_type12}, we can prove Proposition~\ref{presv_bigon_pairs} by the same way as the case of orientable surfaces. 
Proposition~\ref{presv_linked_sharing_pairs} follows from Proposition ~\ref{presv_bigon_pairs}, Lemma~\ref{presv_torus_pairs}, and the following claim: Two bigon pairs $\{c,d\}$ and $\{c',d'\}$ form a linked sharing pair if and only if the following conditions hold:
\begin{enumerate}
    \item each $\{c,d'\}$ and $\{c',d\}$ is an nondegenerate torus pair, and
    \item there is a curve that forms a torus triple with both $\{c,d'\}$ and $\{c',d\}$.
\end{enumerate}
The proof of claim is the same as the case of orientable surfaces. 

\section{Connectedness of fine arc graphs}

In \cite{Long--Margalit--Pham--Verberne--Yao}, they introduced several fine arc graphs of orientable surfaces and proved that they are connected. 
In this section, we consider fine arc graphs of nonorientable surfaces and discuss their connectedness. 
Through of this section, let $F$ be a compact surface with nonempty boundary $\partial F \neq \emptyset$.

We will define the \emph{fine arc graph} $\fa (F)$ of $F$. An arc $a \colon [0,1] \to F$ is said to be \emph{simple} if $a$ is injective,
\emph{proper} if $a^{-1}(\partial F)=\{0,1\}$, and \emph{essential} if it is not homotopic to into $\partial F$. 
We say that two arcs have \emph{disjoint interior} if they are disjoint away from $\partial F$.
The fine arc graph $\fa (F)$ is the graph whose vertices are essential simple proper arcs in $F$ and whose edges connect vertices with disjoint interior. 

In \cite[Proposition 3.1]{Long--Margalit--Pham--Verberne--Yao}, 
they proved that the fine arc graph $\fa(S)$ is connected for every compact orientable surface $S$ with nonempty boundary. Their argument is based on the connectedness of the (ordinary) arc graph $\mathcal{A}(S)$ and the compactness of $S$. These properties also hold for nonorientable surfaces; for every $N=N_g^b$ with $b>0$, the arc graph $\mathcal{A}(N)$ is connected (see \cite[Corollary 1]{Kuno16} for example) and $N$ is compact. Thus, the following proposition holds.

\begin{proposition}[cf. {\cite[Proposition 3.1]{Long--Margalit--Pham--Verberne--Yao}}]
    For any $N=N_g^b$ with $b>0$, the graph $\fa(N)$ is connected.
\end{proposition}

We say that an arc in $F$ is \emph{nonseparating} if its complement in $F$ is connected.
The \emph{fine nonseparating arc graph} $\fna (F)$ is the subgraph of $\fa(F)$ spanned by the nonseparating arcs.

In \cite{Long--Margalit--Pham--Verberne--Yao}, they proved the connectedness of $\fna(S)$ by using the connectedness of $\fa(S)$. 
In their argument, the following claim is a key: if $x$ is a separating arc in $S$ and $R \subset S$ is one
side of $x$, then there is a nonseparating arc $y$ in $S$ that is contained in $R$.
Since this also holds for nonorientable surfaces, we can make the same argument as they do, and thus obtain the following.

\begin{corollary}[cf. {\cite[Cororally 3.2]{Long--Margalit--Pham--Verberne--Yao}}]
    For any $N=N_g^b$ with $b>0$, the graph $\fna(N)$ is connected.
\end{corollary}

For a surface $F$ with positive genus and a connected component $d_0$ of $\partial F$, we will define the \emph{fine linked arc graph} $\fla(F,d_0)$. 
Let $a$ and $b$ be two vertices of $\fa(F)$ whose interiors are disjoint arcs. 
We say that $a$ and $b$ are linked at $d_0$ if all four endpoints lie on $d_0$ and a curve parallel and sufficiently close to $d_0$ intersects the arcs alternately.   
We define $\fla(F,d_0)$ as the graph whose vertices are nonseparating simple proper arcs in $F$ with both endpoints lie on $d_0$ and whose edges connect arcs with disjoint interiors that are linked at $d_0$.

In \cite{Long--Margalit--Pham--Verberne--Yao}, they obtained the connectedness of $\fla(S,d_0)$ from the connectedness of $\fna(S)$. 
The key to their argument is the following claim: if $x$ and $y$ are unlinked, then there exists an arc $z$ which is linked with $x$ and $y$. 
This claim is based on the assumption that $x$ and $y$ are nonseparating, and the same argument goes through for the nonorientable case. 
Hence, we obtain the following.

\begin{corollary}[cf. {\cite[Cororally 3.3]{Long--Margalit--Pham--Verberne--Yao}}] \label{cor:fine_link_arc}
    For any $N=N_g^b$ with $g \geq 1$ and $b>0$, and any component $d_0$ of $\partial N$, the graph $\fla(N,d_0)$ is connected.
\end{corollary}

\section{Automorphisms of the extended fine curve graphs}\label{section_aut_extended_fine_curve_graph}

The goal of this section is to 
provide an outline of the proof of the following:

\begin{theorem}[cf. {\cite[Theorem 1.2]{Long--Margalit--Pham--Verberne--Yao}}]\label{thm:extended_graph}
Let $N$ be a closed nonorientable surface. Then, the natural map $\nu\colon\mathrm{Homeo}(N)\rightarrow {\rm Aut}( \mathcal{EC}^{\dagger}(N) )$ is an isomorphism.
\end{theorem}

We define the {\it extended fine curve graph} $\mathcal{EC}^{\dagger}(N)$ of a closed nonorientable surface $N$ to be a graph whose vertices are all the essential simple closed curves on $N$ or the inessential simple closed curves which bound a disk on $N$, and an edge is a pair of vertices which are disjoint in $N$.
We emphasize that in our argument the vertex set of the extended fine curve graph $\mathcal{EC}^{\dagger}(N)$ of any closed nonorientable surface $N$ does not contain any inessential curves bounding a M\"obius band.

Following the method of \cite{Long--Margalit--Pham--Verberne--Yao}, we use convergent sequences described below to prove Theorem~\ref{thm:extended_graph}. We see that the lemmas and corollaries correspond to \cite[Section 4]{Long--Margalit--Pham--Verberne--Yao} also hold for nonorientable surfaces. 
Since the proofs of the lemmas and corollaries in this section are the same as those of 
\cite{Long--Margalit--Pham--Verberne--Yao} for orientable surfaces, 
we will only describe why they are valid for nonorientable surfaces and omit their proofs.


Let $N$ be a closed nonorientable surface of genus $g$. We say that a sequence of vertices $(c_{i})$ of $\mathcal{EC}^{\dagger}(N)$ {\it converges} to a point $x\in N$ if every neighborhood of $x$ contains all but finitely many of the corresponding curves to $c_{i}$, and we write $\mathrm{lim}(c_{i})=x$. If $(c_{i})$ is convergent, it must be that there exists $M>0$ so that each $c_{i}$ with $i>M$ bounds a disk in $N$.

Since~\cite[Lemma 4.1]{Long--Margalit--Pham--Verberne--Yao} is 
proved by a local argument in a surface, 
the same argument works also for nonorientable surfaces. 

\begin{lemma}[cf. {\cite[Lemma 4.1]{Long--Margalit--Pham--Verberne--Yao}}]\label{connects-the-dots_lemma}
Let $N$ be a closed nonorientable surface, and let $(x_{i})$ be a sequence of points in $N$ that converges to a point $x\in N$. Then, there exists a simple closed curve in $N$ that contains infinitely many of the $x_{i}$. 
\end{lemma}

Since the argument in \cite[Lemma 4.2]{Long--Margalit--Pham--Verberne--Yao} is 
based on the limit point compactness of surfaces,
we can show the same result for nonorientable surfaces.
\begin{lemma}[cf. {\cite[Lemma 4.2]{Long--Margalit--Pham--Verberne--Yao}}]\label{presv_conv_seq}
Let $N$ be a closed nonorientable surface. Automorphisms of $\mathcal{EC}^{\dagger}(N)$ preserve convergent sequences
\end{lemma}

We say that two convergent sequences of vertices of $\mathcal{EC}^{\dagger}(N)$ are {\it coincident} if they converge to the same point of $N$. The {\it interleave} of two sequences $(c_{i})$ and $(d_{i})$ is a sequence $c_{1}$, $d_{1}$, $c_{2}$, $d_{2}$, $\cdots$.
We have the following corollary of Lemma~\ref{presv_conv_seq}.
\begin{corollary}[cf. {\cite[Corollary 4.3]{Long--Margalit--Pham--Verberne--Yao}}]\label{coincidence_conv_seq}
Let $N$ be a closed nonorientable surface. Let $(c_{i})$ and $(d_{i})$ be two convergent sequences of vertices of $\mathcal{EC}^{\dagger}(N)$. Then, $(c_{i})$ and $(d_{i})$ are coincident if the interleave of $(c_{i})$ and $(d_{i})$ is convergent. In particular, automorphisms of $\mathcal{EC}^{\dagger}(N)$ preserve coincidence of convergent sequences. 
\end{corollary}

We say that a sequence $(c_{i}^{1})$, $(c_{i}^{2})$, $(c_{i}^{3})$, $\cdots$ of a convergent sequences of vertices of $\mathcal{EC}^{\dagger}(N)$ {\it converges} if the sequence of limit points ${\rm lim}(c_{i}^{1})$, ${\rm lim}(c_{i}^{2})$, ${\rm lim}(c_{i}^{3})$, $\cdots$ converges to a point $x\in N$. In this case we say that the sequence converges to $x$. 

\begin{corollary}[cf. {\cite[Corollary 4.4]{Long--Margalit--Pham--Verberne--Yao}}]\label{presv_conv_seq_conv_seq}
Let $N$ be a closed nonorientable surface. Let $(c_{i})$ and $(d_{i})$ be two convergent sequences of vertices of $\mathcal{EC}^{\dagger}(N)$. Then, $(c_{i})$ and $(d_{i})$ are coincident if and only if the interleave of $(c_{i})$ and $(d_{i})$ is convergent. In particular, automorphism of $\mathcal{EC}^{\dagger}(N)$ preserve coincidence of convergent sequences. 
\end{corollary}

We say that a vertex $c$ is a {\it limit curve} for a sequence $(c_{i})$ of vertices $\mathcal{EC}^{\dagger}(N)$ if ${\rm lim}(c_{i})\in c$. 
\begin{corollary}[cf. {\cite[Corollary 4.5]{Long--Margalit--Pham--Verberne--Yao}}]\label{presv_limit_curve}
Let $N$ be a closed nonorientable surface. Let $c$ be a limit curve for a sequence $(c_{i})$ of vertices and $\alpha$ an automorphism of $\mathcal{EC}^{\dagger}(N)$. Then, $\alpha(c)$ is a limit curve for $(\alpha(c_{i}))$.
\end{corollary}

We provide an outline of the proof of Theorem~\ref{thm:extended_graph}. 
We can prove the injectivity of $\nu$ as follows.

\begin{lemma}[cf. {\cite[Lemma 4.6]{Long--Margalit--Pham--Verberne--Yao}}]\label{inj_extended_fine_graph}
Let $N$ be a closed nonorientable surface of genus $g$. The natural map $\nu\colon \mathrm{Homeo}(N)\rightarrow\mathrm{Aut}(\mathcal{EC}^{\dagger}(N))$ is injective.
\end{lemma}

\begin{proof}
    We can prove same as in \cite{Long--Margalit--Pham--Verberne--Yao} as follows: Assume that $f \in \operatorname{Ker} \nu$. For every $x \in N$, there exist two curves $c,d \in \mathcal{C}^{\dagger}(N)$ such that $c \cap d = \{x\}$. Then $f(x)=f(c \cap d)=c \cap d = x$ and this implies that $f$ is the identity.
\end{proof}

Therefore it is sufficient to prove that $\nu$ is surjective. As in \cite{Long--Margalit--Pham--Verberne--Yao}, we can define
$\xi \colon {\rm Aut}(\mathcal{EC}^{\dagger}(N)) \to {\rm Homeo}(N)$ as follows: For $\alpha \in {\rm Aut}(\mathcal{EC}^{\dagger}(N))$, we define $\xi(\alpha) \in {\rm Homeo}(N)$ by $\xi(\alpha)(x)=\lim ( \alpha(c_i))$, where $(c_i)$ is a convergent sequence of $x \in N$.

As Claims 1 and 2 in the proof of \cite[Lemma 4.6]{Long--Margalit--Pham--Verberne--Yao}, we can observe the following:
\begin{enumerate}[(a)]
    \item For $\alpha \in {\rm Aut}(\mathcal{EC}^{\dagger}(N))$  and $c \in \mathcal{EC}^{\dagger}(N), \; \xi(\alpha)(c)=\alpha(c)$. \label{claim_a}
    \item For $f \in  {\rm Homeo}(N)$ and $c \in \mathcal{EC}^{\dagger}(N), \nu(f)(c)=f(c)$. \label{claim_b} 
\end{enumerate}
By using \ref{claim_a} and \ref{claim_b}, for every $\alpha \in {\rm Aut}(\mathcal{EC}^{\dagger}(N))$ and every $c \in \mathcal{EC}^{\dagger}(N)$, we have
\[ \nu(\xi(\alpha))(c)=\xi(\alpha)(c)=\alpha(c). \]
This implies that $\nu(\xi(\alpha))= \alpha$ and thus $\nu$ is surjective; Theorem~\ref{thm:extended_graph} holds.

\section{Automorphisms of the fine curve graph}

In this section, we prove Theorem~\ref{thm:main}. As same as Lemma~\ref{inj_extended_fine_graph}, the following lemma holds. 

\begin{lemma}[cf. {\cite[Lemma 5.1]{Long--Margalit--Pham--Verberne--Yao}}]\label{inj_fine_graph}
Let $N$ be a closed nonorientable surface. The natural map $\eta\colon \mathrm{Homeo}(N)\rightarrow\mathrm{Aut}(\mathcal{C}^{\dagger}(N))$ is injective.
\end{lemma}

Before proving Theorem~\ref{thm:main}, we prepare a tool to prove it. Let $\Gamma$ be a graph and $\Delta\subset\Gamma$ a subgraph. A map $E\colon\mathrm{Aut}(\Delta)\rightarrow\mathrm{Aut}(\Gamma)$ is an {\it extension map} if for any $\varphi\in\mathrm{Aut}(\Delta)$, $E(\varphi)(\Delta)=\Delta$ and $E(\varphi)|_{\Delta}=\varphi$. 

From now on, we prove Theorem~\ref{thm:main}.
We follow the same strategy as in \cite{Long--Margalit--Pham--Verberne--Yao} for the proof.

\begin{proof}[Proof of Theorem~\ref{thm:main}]
We prove in two steps.
In the first step, we construct an extension homomorphism $\varepsilon\colon\mathrm{Aut}(\mathcal{C}^{\dagger}(N))\rightarrow\mathrm{Aut}(\mathcal{EC}^{\dagger}(N))$, and in the second step, we complete the proof by using the extension homomorphism $\varepsilon$.

{\it Step 1.}
Let $\alpha\in\mathrm{Aut}(\mathcal{C}^{\dagger}(N))$. We define $\hat{\alpha}\colon\mathcal{EC}^{\dagger}(N)\rightarrow\mathcal{EC}^{\dagger}(N)$ as follows. If $c$ is an essential curve on $N$, that is, if $c\in\mathcal{C}^{\dagger}(N)$, then $\hat{\alpha}(c)=\alpha(c)$. If $e$ is an inessential curve which bounds a disk in $N$, then we choose a nonseparating bigon pair $\{c,d\}$ determines $e$ and we correspond $\hat{\alpha}(e)$ to an inessential curve bounding a disk determined by the nonseparating bigon pair $\{\alpha(c), \alpha(d)\}$; this correspondence makes sense because of Proposition~\ref{presv_bigon_pairs}. We will verify $\hat{\alpha}$ is well-defined and 
an automorphism of $\mathcal{EC}^{\dagger}(N)$.

Suppose that $\{c',d'\}$ is another nonseparating bigon pair determines $e$. It follow from Corollary \ref{cor:fine_link_arc} that there exists a sequence of bigon pairs 
\begin{equation*}
\{c,d\}=\{c_{0},d_{0}\}, \{c_{1}, d_{1}\},\cdots, \{c_{n}, d_{n}\}=\{c',d'\},
\end{equation*}
where each pair $\{\{c_{i},d_{i}\}, \{c_{i+1}, d_{i+1}\}\}$ is a linked sharing pair for $e$. It follows from Proposition~\ref{presv_linked_sharing_pairs} that $\{\alpha(c'), \alpha(d')\}$ also determines the inessential curve $\hat{\alpha}(e)$ bounding a disk, and we see $\hat{\alpha}$ is well-defined.

Next we will show that $\hat{\alpha}$ is an automorphism.
Since $\hat{\alpha}$ is bijection on the set of vertices of $\mathcal{EC}^{\dagger}(N)$ it suffices to show that $\hat{\alpha}$ maps any edge to an edge. 

Let $a_{1}$ and $a_{2}$ are two distinct essential curves. Then $\hat{\alpha}(a_{1})=\alpha(a_{1})$ and $\hat{\alpha}(a_{2})=\alpha(a_{2})$, and so $\hat{\alpha}(a_{1})$ and $\hat{\alpha}(a_{2})$ form an edge if and only if $a_{1}$ and $a_{2}$ form an edge.

For the case where an essential curve $c$ and an inessential curve $e$ bounding a disk form an edge, we claim that we can take a nonseparating bigon pair $\{d_{1}, d_{2}\}$ which determines $e$ and disjoint from $c$. 
In fact, if $c$ is a separating curve on $N$, we set the two connected components $N'$ and $N''$ of $N\setminus c$, and we suppose that $e$ lies on $N'$.
Since $c$ is essential and $N$ is a closed nonorienatble surface, the genera of $N'$ and $N''$ are at most $2$. Hence we can take two-sided (note that bigon pairs are constructed by only two-sided curves) nonseparating curves $d_{1}$ and $d_{2}$ on the component $N'$ containing $e$ so that the pair $\{d_{1}, d_{2}\}$ is a bigon pair which determines $e$. If $c$ is a nonseparating curve on $N$, the complement $N\setminus c$ of $c$ in $N$ is homeomorphic to any one of $N_{g-1}^{1}$, $N_{g-2}^{2}$, $S_{\frac{g-1}{2}}^{1}$, or $S_{\frac{g-2}{2}}^{2}$. Since we assume that the genus of $N$ is at least $4$, we can take two-sided nonseparating curves $d_{1}$ and $d_{2}$ on the subsurface $N\setminus c$ so that the pair $\{d_{1}, d_{2}\}$ is a bigon pair which determines $e$, as desired.

\begin{figure}[h]
\begin{center}
\includegraphics[scale=0.75]{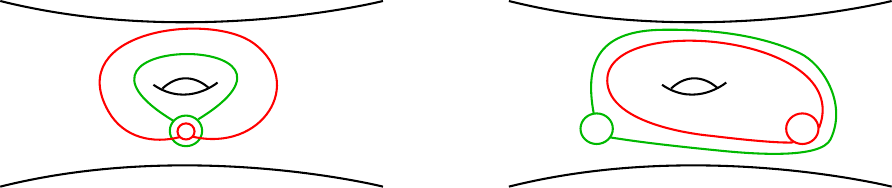}
\caption{Bigon pair for disjoint inessential curves bounding a disk in the nested case (left) and the unnested case (right).}\label{fig_two_bigon_pairs_v2}
\end{center}
\end{figure}

For the case where two inessential curves $e$ and $f$ which bound a disk form an edge, similar to the case of orientable surfaces we claim that two inessential curves $e$ and $f$ bounding a disk are disjoint if and only if the following holds up to relabeling $e$ and $f$: for every nonseparating bigon pair $\{c,d\}$ that determines $e$, there is a bigon pair $\{c',d'\}$ that determines $f$ and is disjoint from $e$. The forward direction is proved by direct construction (see Figure~\ref{fig_two_bigon_pairs_v2}), and note that since the genus of $N$ is at least $4$, $N$ has a subsurface homeomorphic to an orientable subsurface of genus $1$.
For the reverse direction, we assume that two inessential curves $e$ and $f$ which bound a disk intersect. We choose an intersection point $x\in e\cap f$. We can take a nonseparating bigon pair $\{c,d\}$ determining $e$ where $x$ is one of the vertices of the bigon. Let $\{c',d'\}$ be any bigon pair which determines $f$. Then $x\in f\subset c'\cup d'$, and so $c'\cup d'$ intersects $c$, as desired. By this claim, we see that $\hat{\alpha}$ preserves the set of edges spanned by two inessential curves bounding a disk. Therefore we see that $\hat{\alpha}$ is an automorphism of $\mathcal{EC}^{\dagger}(N)$, in particular $\hat{\alpha}$ is an automorphism of $\mathcal{C}^{\dagger}(N)$. 

By the definition of $\varepsilon\colon\mathrm{Aut}(\mathcal{C}^{\dagger}(N))\rightarrow\mathcal{EC}^{\dagger}(N)$ given by $\varepsilon(\alpha)=\hat{\alpha}$ is the desired extension map.

{\it Step 2.}
Recall that $\eta\colon\mathrm{Homeo}(N)\rightarrow\mathrm{Aut}(\mathcal{C}^{\dagger}(N))$ and $\nu\colon\mathrm{Homeo}(N)\rightarrow\mathrm{Aut}(\mathcal{EC}^{\dagger}(N))$ are the natural homomorphisms. By Theorem~\ref{thm:extended_graph}, $\nu$ is an isomorphism. Let $\varepsilon$ be the extension homomorphism constructed in the first step. 

As with \ref{claim_b} in the previous section, the following holds:
\begin{enumerate}[(a)]
\setcounter{enumi}{2}
    \item For $f \in  {\rm Homeo}(N)$ and $c \in \mathcal{C}^{\dagger}(N)$, $\eta(f)(c)=f(c)$. \label{claim_c}
\end{enumerate}
By \ref{claim_a} and \ref{claim_c}, for every $\alpha \in {\rm Aut}(\mathcal{C}^{\dagger}(N))$ and every $c \in \mathcal{C}^{\dagger}(N)$, we have
\[ \eta(\xi(\varepsilon(\alpha)))(c)=\xi(\varepsilon(\alpha))(c)=\varepsilon(\alpha)(c)=\alpha(c). \]
This implies that $\eta(\xi(\varepsilon(\alpha)))= \alpha$ and thus $\eta$ is surjective.
Since $\eta$ is injective by Lemma~\ref{inj_fine_graph}, 
we have finished the proof.
\end{proof}

\begin{remark}\label{remark_for_extended_graph}
In the definition of $\mathcal{EC}^{\dagger}(N)$, we excluded curves that bound a M\"{o}bius band and allowed only those that bound a disk as inessential curves. 
This is done to avoid the difficulty of including curves that bound a M\"{o}bius band, and it is sufficient for the proof of our main theorem.

\end{remark}

\begin{remark}\label{remark_for_genus}
In the part where we verify $\hat{\alpha}$ is bijection for the edges between one essential curve and one inessetial curve bounding a disk in Proof of Theorem~\ref{thm:main}, if the genus of a closed nonorientable surface is $g\leq 3$, we can not take a bigon pair $\{d_{1},d_{2}\}$ which is disjoint from $c$ in the case where $c$ is nonseparating. In fact, in our definition of bigon pairs $\{d_{1},d_{2}\}$, the curves $d_{1}$ and $d_{2}$ are two-sided curves, and by Proposition~\ref{presv_bigon_pairs} we see that automorphisms of the fine curve graph $\mathcal{C}^{\dagger}(N)$ preserve ``nonseparating'' bigon pairs. Therefore, we should take two-sided and nonseparating curves $d_{1}$ and $d_{2}$ on $N\setminus c$ which is homeomorphic to $N_{g-1}^{1}$, $N_{g-2}^{2}$, $S_{\frac{g-1}{2}}^{1}$, or $S_{\frac{g-2}{2}}^{2}$. If $g\leq 3$, $N\setminus c$ can be $N_{1}^{1}$, $N_{1}^{2}$, $S_{0}^{1}$, or $S_{0}^{2}$, and we can not take two-sided essential curves in these surfaces (see~\cite[Section 2.4]{Atalan--Korkmaz14} for example).
\end{remark}

\section*{Acknowledgment}
The authors wish to express their great appreciation to Genki Omori for valuable comments.
The first author is supported by 
JST-Mirai Program Grant Number JPMJMI22G1.
The second author is supported by JSPS KAKENHI, grant number 21K13791.

\end{document}